\documentclass[10pt]{article}
\usepackage{graphicx} 
\usepackage{amsthm,amsfonts,amssymb,amsmath,oldgerm}
\usepackage{fullpage}
\usepackage{graphicx}
\usepackage{mathrsfs}
\usepackage{xcolor}
\numberwithin{equation}{section}
\usepackage{enumitem}
\usepackage{hyperref}
\usepackage{subcaption}

\makeatletter
\def\namedlabel#1#2{\begingroup
    #2%
    \def\@currentlabel{#2}%
    \phantomsection\label{#1}\endgroup
}
\makeatother

\newcommand{\R}{\mathbb R}

\renewcommand{\Re}{\mathrm{Re}}

\newcommand\Lip{\textrm{Lip}}
\newcommand\Dini{\rm Din}
\newcommand\Hloc{H_{\textrm{loc}}}

\newtheorem{theorem}{Theorem}[section]
\newtheorem{proposition}[theorem]{Proposition}
\newtheorem{corollary}[theorem]{Corollary}
\newtheorem{lemma}[theorem]{Lemma}
\newtheorem{remark}[theorem]{Remark}

\allowdisplaybreaks[3]

\title{Uniformity in the Fourier inversion formula with applications to Laplace transforms}
\author{Joannis Alexopoulos\thanks{Department of Mathematics, Karlsruhe Institute of Technology, Englerstra\ss e 2, 76131 Karlsruhe, Germany;\\ \texttt{joannis.alexopoulos@kit.edu}}}

\begin{document}
\maketitle
\begin{abstract}
We systematically find conditions which yield locally uniform convergence in the Fourier inversion formula in one and higher dimensions. We apply the gained knowledge to the complex inversion formula of the Laplace transform to extend known results for Banach space-valued functions and, specifically, for $C_0$-semigroups. 
\paragraph*{Keywords.} Banach space-valued functions; Fourier transform; Laplace transform; complex inversion formula; uniform convergence; local H\"older-continuity; strongly continuous semigroups; Favard spaces\\
\textbf{Mathematics Subject Classification (2020).} 44A10; 47D06; 26A16
\end{abstract}
\section{Introduction}

We fix a Banach space $X$ for the entire manuscript.

\bigskip

\noindent For an exponentially bounded function $F:\R_+:=[0,\infty) \rightarrow X$ with exponential growth bound $\omega_0 \in \R$, i.e. there exists $C>0$ such that $||F(t)|| \leq Ce^{\omega_0 t}$ for all $t \in \R_+$, the Laplace transform  is given by
\begin{align}
\label{laplace}
\mathcal{L}(F)(\lambda) := \int_0^\infty e^{- \lambda s} F(s) \,ds \text{, } \quad \Re( \lambda) >\omega_0.
\end{align}
For many decades, mathematicians have been interested in how we can formally and qualitatively invert (\ref{laplace}). The importance of this question is not only raised by abstract curiosity. Nowadays, the Laplace transform is a standard and important tool to analyze differential equations: having properties of the transformed solution of a differential equation at hand, the validity of a suitable inversion formula may give further information about the solution itself. A prominent realization of this principle is the Hille-Yosida theorem in the context of linear evolution equations.

Let $\omega>\omega_0$. The complex inversion formula is defined by
\begin{align}
\label{inverse_def}
 \lim_{R \rightarrow \infty} \frac{1}{2\pi i} \int_{\omega - i R}^{\omega + i R} e^{\lambda t} \mathcal{L}(F)(\lambda) \,d\lambda
\end{align}
for $t \in \R_+$. We are interested under which regularity or localization conditions on $F$, we can guarantee
\begin{align}
\label{formula_of_interest}
F(t) = \lim_{R \rightarrow \infty} \frac{1}{2\pi i}  \int_{\omega - i R}^{\omega + i R} e^{\lambda t} \mathcal{L}(F)(\lambda) \,d\lambda \textrm{ locally uniformly in } t\in \R_+.
\end{align}
In  this paper, we consider general Banach space-valued functions and prove

\begin{proposition}
\label{prop_h_loc}
    If $F \in \Hloc(\R_+;X)$\footnote{For the definitions of function spaces see Appendix \ref{appendix}}, with  $F(0) = 0$, has exponential growth bound $\omega_0 \in \R$, then
\begin{align*}
F(t) =  \lim\limits_{R\rightarrow \infty} \frac{1}{2\pi i} \int\limits_{\omega - i R}^{\omega+iR} e^{\lambda t} \mathcal{L}(s)(\lambda) \,d\lambda \textrm{ locally uniformly in } t\in \R_+,
\end{align*}
for any $\omega>\omega_0$. 
\end{proposition}
Before discussing the proof of Proposition \ref{prop_h_loc}, we give a short overview of existing results providing answers to (\ref{formula_of_interest}) in the current literature. 
For strongly continuous functions, the question \textit{under which conditions the complex inversion formula holds and in which sense} was posed more generally and was partially answered by Markus Haase in \cite{haase}. Haase considers convolutions of a strongly continuous and a locally integrable scalar function and obtains (\ref{formula_of_interest}) with convergence in operator norm. Another recent result can be found in \cite[Appendix A]{ours}. The authors show the validity of Laplace inversion for convolutions of two $C_0$-semigroups as a direct consequence of the upcoming Proposition \ref{know_result}. This observation is then exploited to identify and control the leading-order terms in the Neumann series of high-frequency components of a $C_0$-semigroup.
%To give a comparison to Proposition \ref{prop_h_loc}, we now collect results addressing general Banach space-valued functions. 
For merely continuous functions, we cite the result \cite[Theorem 4.2.21 b)]{arendt} which requires introducing the (weaker) Ces\`{a}ro limit: 
\begin{align*}
\textrm{ C} - \lim\limits_{R \rightarrow \infty} \int\limits_{\omega - i R}^{\omega + i R} e^{\lambda t} \mathcal{L}(F)(\lambda) \,d\lambda := \lim\limits_{R \rightarrow \infty} \frac{1}{R} \int_0^R \int\limits_{\omega - i r}^{\omega + i r} e^{\lambda t} \mathcal{L}(F)(\lambda) \,d\lambda \,dr,\quad t\in \R_+.
\end{align*}
\begin{proposition}
\label{cesaro_result}
For $F \in C(\R_+;X)$, $F(0) = 0$, with exponential growth bound $\omega_0$, we have $$F(t) =  \mathrm{ C} -  \frac{1}{2\pi i} \lim\limits_{R \rightarrow \infty} \int\limits_{\omega - i R}^{\omega + i R} e^{\lambda t} \mathcal{L}(F)(\lambda) \,d\lambda \textrm{ locally uniformly in } t \in \R_+,$$ for any $\omega>\omega_0$.
\end{proposition}

The next result is, to the author's knowledge, the most general result providing (\ref{formula_of_interest}) in the current literature.
\begin{proposition}
\label{know_result}
If $F \in \mathrm{Lip}(\R_+;X)$, $F(0) = 0$, then
\begin{align*}
F(t) = \lim\limits_{R\rightarrow \infty} \frac{1}{2\pi i} \int\limits_{\omega - i R}^{w+iR} e^{\lambda t} \mathcal{L}(s)(\lambda) \,d\lambda \textrm{ locally uniformly in } t\in \R_+,
\end{align*}
 for any $\omega>0$. 
\end{proposition}
The proof of Proposition \ref{know_result} in \cite{arendt} relies on \cite[Theorem 2.3.4]{arendt}, which is obtained through the Riesz-Stieltjes Representation \cite[Theorem 2.1.1]{arendt}. This representation gives an abstract characterization of $\{F \in \Lip(X): F(0) = 0\}$. The authors give a second proof of Proposition \ref{know_result} in \cite[page 260]{arendt}, based on the existence of the Ces\`{a}ro limit, c.f. Proposition \ref{cesaro_result}, together with the property
\begin{align}
\label{feebly_osc}
\left|\left|t \mapsto \int\limits_{\omega - i R}^{\omega + i R} e^{\lambda t} \mathcal{L}(F(t))(\lambda) \,d\lambda- \int\limits_{\omega - i S}^{\omega + i S} e^{\lambda t} \mathcal{L}(F(t))(\lambda) \,d\lambda\right|\right|_{C([0,a],X)} \rightarrow 0  \textrm{ as } S,R\rightarrow \infty, \quad \frac{R}{S} \rightarrow 1
\end{align}
for all $a>0$, i.e. $F$ is \textit{feebly oscillating}. %We say $\lambda \rightarrow e^{\lambda \cdot} \mathcal{L}(\cdot) \,d\lambda$ is feebly oscillating with respect to $\Lip(X)$. 
Justifying both properties is sufficient to conclude (\ref{formula_of_interest}) by the real Tauberian theorem \cite[Theorem 4.2.5]{arendt}. 
\begin{remark}
    An example which is covered by Proposition \ref{prop_h_loc} but not by Proposition \ref{know_result}, can be obtained by considering a Weierstrass function, for example $G(t) = \sum_{n = 1}^\infty \frac{\sin(n^2t)}{n^2}$, and setting $F(t) = e^{-t}G(t)$.
\end{remark}
Due to the observation, see Section \ref{identify}, that for exponentially bounded functions, Laplace inversion in the sense of (\ref{inverse_def}) can be traced back to  Fourier inversion, we study the Fourier transform first and provide sufficient conditions yielding locally uniform convergence of the Fourier inversion formula in one and higher dimensions. The gained knowledge on the Fourier inversion formula is then exploited to prove Proposition \ref{prop_h_loc}. %$F \in \Hloc(\R_+;X)$. 
We conclude this paper with an application of Proposition \ref{prop_h_loc} to $C_0$-semigroups on Favard spaces.

\paragraph*{Acknowledgments.}  The author is grateful to the anonymous referee for his very helpful suggestions, in particular, for the advice to study first the Fourier inversion formula and then to apply the knowledge to the Laplace transform.
This project is funded by the Deutsche Forschungsgemeinschaft (DFG, German Research Foundation) -- Project-ID 491897824.

\section{Uniformity in the Fourier inversion formula}
\subsection{Preliminaries}
We start with recalling classical results and formulae related to Fourier inversion for Banach space-valued functions. 
Fix $n\geq 1$ and let $F: \R^n \rightarrow X$ be $L^1$-integrable.
Define the Fourier transform
\begin{align*}
    \mathcal{F}(F)(k) := \hat{F}(s) := \int_{\R^n} e^{-ik\cdot s} F(s)\,ds,  \quad k \in \R^n
\end{align*}
and the inversion formula is given by
\begin{align*}
    \mathcal{F}^{-1}(\hat{F})(x) = \lim_{R \rightarrow \infty} S_R(F)(x), \quad x \in \R^n
\end{align*}
with 
\begin{align*}
    S_R(F)(x) = \frac{1}{(2\pi)^n}\int_{|k| \leq R} e^{ix \cdot k} \hat{F}(k) \,dk, \quad, \quad R\geq 1,\quad x \in \R^n.
\end{align*}
Analogously to the Laplace transform, we ask for regularity and localization assumptions such that
\begin{align}
\label{locally_uniform_intro_Fourier}
    S_R(F)(x) \rightarrow F(x) \textrm{ as } R\rightarrow \infty,  \text{ locally uniformly in } x \in \R^n.
\end{align}
If (\ref{locally_uniform_intro_Fourier}) holds, we shortly write: \textit{Fourier inversion holds locally uniformly}.
A central role in the study of Fourier inversion plays the spherical mean of $F$ at $x \in \R^n$ defined by the surface integral
\begin{align*}
    \bar{F}_x(r) = \frac{1}{\omega_{n-1}} \int_{\mathbb{S}^{n-1}} F(x+ r \omega) \,d\omega, \quad r\geq 0,
\end{align*}
where $\mathbb{S}^{n-1}$ is the unit sphere in $\R^n$ and $\omega_{n-1}$ is its surface measure. 
Following \cite{pinsky_book}, we rewrite
\begin{align}
\label{formula_invert}
\begin{split}
      S_R(F)(x) &= \int_{\R^n} F(x+z) D^R_n(z) \,dz \\
    &= \omega_{n-1} \int_0^\infty r^{n-1} D^R_n(r)  \bar{F}_x(r)\,dr
\end{split}
\end{align}
with the Dirichlet kernel $D^R_n(r)$ given by\footnote{Denote by $\widetilde{D}_R^n$ the Dirichlet kernel given in \cite{pinsky_book}. Then, the relation $D^R_n(z) = (2\pi)^{-n} \widetilde{D}_R^n((2\pi)^{-1} z)$ holds. This is due to slightly different choices of the Fourier transforms.}
\begin{align*}
    D^R_n(z) = \frac{1}{(2\pi)^n}\int_{|\xi|\leq R} e^{-iz\cdot\xi} \,d\xi, \quad R\geq 1,\quad z \in \R^n.
\end{align*}
Note that for a rotation matrix $D \in \R^{n\times n}$, it holds $(Dz) \cdot \xi = z \cdot (D^T\xi)$ and, since $D^T$ is also a rotation matrix, the substitution rule first gives $D^R_n(z) = D^R_n(|z|)$ and then the second identity in (\ref{formula_invert}) follows.

To that end, we record the following facts.
\begin{proposition}[\cite{pinsky_book}]
\label{prop_properties}
    \begin{enumerate}
    \item We have
    \begin{align*}
        D_1^R(r) = \frac{\sin(Rr)}{\pi r}, \quad D_2^R(r) = \frac{RJ_1(R r)}{2\pi r}, \quad RJ_1(R r) = -\frac{d}{dr} J_0(R r), 
    \end{align*}
   for $r>0$ and $R\geq 1$, where $J_0$ and $J_1$ denote Bessel functions of first kind.  In particular, $J_0$ and $J_1$ are smooth and bounded.
    \item For every $\delta>0$,
\begin{align*}
    \int_{-\delta}^\delta \frac{\sin(Rs)}{s} \,ds \rightarrow \pi, \quad  \int_{0}^\delta \frac{\sin(Rs)}{s} \,ds \rightarrow \frac{\pi}{2} \text{ as } R \rightarrow \infty.
\end{align*}
\item It holds
\begin{align*}
    J_0(0) = 1, \quad J_0(r) \rightarrow 0 \text{ as } r \rightarrow \infty.
\end{align*}
\item 
For $n\in \mathbb{N}_{\geq 3}$, we have the relation
\begin{align*}
     D_n^R(r) = -\frac{1}{r} \frac{d}{dr} D_{n-2}^R(r), \quad   r>0, \quad R\geq 1.
\end{align*} 
\item For each $n\in \mathbb{N}$ there exists some $C_n>0$ such that
\begin{align*}
    |D_n^R(r)| \leq  \frac{C_n R^n}{(1+Rr)^{\frac{n+1}{2}}}, \quad  r>0, \quad R\geq 1.
\end{align*}
\end{enumerate}
\end{proposition}
We distinguish between odd and even dimensions and study the cases $n = 1$ and $n= 2$ independently.
The following formulae are a consequence of an iterative application of Proposition \ref{prop_properties} (iv).
\begin{lemma}
\label{formula_lemma}
Let $(\eta_R)_{R\geq 1}$ be a sequence of smooth cut-off functions with $0\leq \eta_R\leq 1$ being one in the interval $(0,R)$ and zero in $(R+1,\infty)$.  We obtain the following identities.
    \begin{enumerate}
        \item Let $n = 2k+1$ with $k\in \mathbb{N}$. Let $B: \R^n \rightarrow X$ be $L^1$-integrable and $k$-times continuously differentiable. There exist universal constants $C_{0,k}(n),...,C_{k,k}(n) \in \R$ such that
        \begin{align*}
            \int_0^\infty D_n^R(r) B(r) r^{n-1} \,dr &=\sum_{j = 0}^k C_{j,k}(n)  \int_0^\infty D_1^R(r) r^{j} \left(\frac{d}{dr}\right)^j (B(r)\eta_R(r)) \,dr \\
            &\quad + \int_0^\infty D_n^R(r) B(r) (1-\eta_R(r)) r^{n-1} \,dr.
        \end{align*}
        \item Let $n = 2k$ with $k\in \mathbb{N}$. Let $B: \R^n \rightarrow X$ be integrable and ($k-1$)-times continuously differentiable.  There exist universal constants $C_{0,k-1}(n),...,C_{k-1,k-1}(n) \in \R$ such that
        \begin{align*}
            \int_0^\infty D_n^R(r) B(r) r^{n-1} \,dr &=  \sum_{j = 0}^{k-1} C_{j,k-1}(n) \int_0^\infty D_2^R(r)r^{j+1} \left(\frac{d}{dr}\right)^j (B(r)\eta_R(r)) \,dr\\
            &\quad + \int_0^\infty D_n^R(r) B(r) (1-\eta_R(r)) r^{n-1} \,dr.
        \end{align*}
    \end{enumerate}
\end{lemma}
\begin{proof}
    For (i), we refer to \cite[pages 141-142]{{pinsky_book}}.
    We prove (ii) for the sake of convenience as we did not find a proof for this formula. For this purpose, we show for every $k \in \mathbb{N}$ with $n = 2k$, that it holds
    \begin{align}
    \label{formula_by_induction}
        \int_0^\infty D_n^R(r) A(r) r^{n-1} \,dr &=\sum_{j = 0}^{k-1} C_{j,k-1}(n) \int_0^\infty D_2^R(r)r^{j+1} \left(\frac{d}{dr}\right)^j (A(r)) \,dr
    \end{align}
    for all compactly supported and ($k-1$)-times differentiable $A: (0,\infty) \rightarrow X$.
    For $k = 1$, this is precisely formula (\ref{formula_invert}) with $n = 2$. We do an induction argument and therefore suppose that the assertion (\ref{formula_by_induction}) is known for $k-1$ with $k\geq 2$. Let $A: (0,\infty) \rightarrow X$ be any compactly supported and ($k-1$)-times differentiable function. There exist universal constants $C_j^1 \in \R$ and $C_j^2 \in \R$ for $j = 0,...,k-2$ such that 
    \begin{align}
    \label{leibnitz_type}
        \left(\frac{d}{dr}\right)^j \left(r \frac{d}{dr}A(r)\right) = C_j^1 r \left(\frac{d}{dr}\right)^{j+1} (A(r)) + C_j^2\left(\frac{d}{dr}\right)^{j} (A(r)).
    \end{align}
    Integrating by parts once and using Proposition \ref{prop_properties}, we find
    \begin{align*}
        \int_0^\infty D_n^R(r) A(r) r^{n-1} \,dr  &= \int_0^\infty \left(-\frac{1}{r} \frac{d}{dr}\right)^{k-1}\left(D_2^R(r)\right) A(r)r^{n-1} \,dr \\
        &= \int_0^\infty \underbrace{\left(-\frac{1}{r} \frac{d}{dr}\right)^{k-2} \left(D_2^R(r) \right)}_{ = D_{n-2}^R(r)} \left(r \frac{d}{dr}\left(A(r)\right)\ + (n-2)A(r) \right) r^{n-3} \,dr.
    \end{align*}
    Applying the induction hypothesis (\ref{formula_by_induction}) to $A(r)$ and $r\frac{d}{dr}(A(r))$ and using (\ref{leibnitz_type}), it holds
    \begin{align*}
         \int_0^\infty D_n^R(r) A(r) r^{n-1} \,dr &= (n-2)\sum_{j = 0}^{k-2}C_{j,k-2}(n-2) \int_0^\infty D_2^R(r)r^{j+1} \left(\frac{d}{dr}\right)^j (A(r))\,dr \\
         &\quad + \sum_{j = 0}^{k-2}C_j^1 C_{j,k-2}(n-2) \int_0^\infty D_2^R(r)r^{j+2} \left(\frac{d}{dr}\right)^{j+1} (A(r)) \,dr \\
         &\quad +\sum_{j = 0}^{k-2}C_j^2 C_{j,k-2}(n-2)  \int_0^\infty D_2^R(r)r^{j+1} \left(\frac{d}{dr}\right)^{j} (A(r)) \,dr.
    \end{align*}
   Since $A$ was chosen arbitrarily,  this gives the assertion (\ref{formula_by_induction}) for $n = 2k$  with determinable constants $C_{0,k-1}(n),..., C_{k-1,k-1}(n) \in \R$. Splitting
    \begin{align*}
        \int_0^\infty D_n^R(r) B(r) r^{n-1} \,dr &=  \int_0^\infty D_n^R(r) (B(r) \eta_R(r)) r^{n-1} \,dr\\
            &\quad + \int_0^\infty D_n^R(r) B(r) (1-\eta_R(r)) r^{n-1} \,dr, 
    \end{align*}
    the claimed formula (ii) follows by choosing $A(r) = B(r) \eta_R(r)$ in (\ref{formula_by_induction}).
\end{proof}
Finally, we isolate the part of the inversion formula which tends to zero uniformly for every $x \in \R^n$.
\begin{lemma}
\label{lemma_universel_zero}
 For every $n\in \mathbb{N}$, there exists some constant $C_n>0$ such that
        \begin{align*}
            \sup_{x \in \R^n}\left|\left|\int_0^\infty D_n^R(r) \bar{F}_x(r) (1-\eta_R(r)) r^{n-1} \,dr \right|\right| \leq \frac{C_n}{R} ||F||_{L^1}
        \end{align*}
for every $R\geq 1$ and $F \in L^1(\R;X)$.    
\end{lemma}
\begin{proof}
     We invoke Proposition \ref{prop_properties} (v) to infer
    \begin{align*}
        \left|\left|\int_0^\infty D_n^R(r) \bar{F}_x(r) (1-\eta_R(r)) r^{n-1} \,dr\right|\right| \leq  \left|\left|\int_R^\infty D_n^R(r) \bar{F}_x(r) r^{n-1} \,dr\right|\right| \leq \frac{C_n}{R} \int_R^\infty r^{n-1} |\bar{F}_x(r)| \,dr \leq \frac{C_n}{R} ||F||_{L^1},
    \end{align*}
     where we use that $\int_0^\infty r^{n-1}  \int_{\mathbb{S}^{n-1}} |F(x+r\omega)|\,d\omega\,dr = ||F||_{L^1}$ for every $x \in \R^n$.
    We also refer to \cite[page 141]{pinsky_book}.
\end{proof}
We note that in one spatial dimension, i.e. for $n = 1$, it holds \begin{align}
\label{formula_one_dim}
    \bar{F}_x(r) = \frac{1}{2}(F(x+r)+F(x-r)), \quad \omega_0 = 2.
\end{align}
Invoking \cite[Theorem 2.3.4]{pinsky_book} and a vector-valued version of the Riemann-Lebesgue lemma, see e.g. \cite{AnaBanach}, one finds that the additional condition  $F \in \Dini(\R;X)$ implies pointwise Fourier inversion. Pinsky also mentions local H\"older-continuity as a sufficient condition for pointwise Fourier inversion in \cite[Corollary 2.3.5]{pinsky_book}. In the upcoming section, we lift this pointwise convergence to locally uniform convergence.

\subsection{Compactness criteria and result on \texorpdfstring{$\R$}{Lg}}
In this section, we consider $n = 1$.
The following compactness criteria are sketched in \cite[Corollary 2]{pinsky_paper}. For the sake of self-containment, we provide their proofs here.
\begin{lemma}[A uniform version of the Riemann-Lebesgue lemma]
\label{uniform_riemann}
   Let $N$ be any set. Let $F: N\times \R \rightarrow X$. If
    \begin{align*}
        \{s \mapsto F(t,s): t\in N\}
    \end{align*}
    is a compact subset of $L^1(\R;X)$, then
    \begin{align*}
        \int_\R \sin(Rs)F(t,s) \,ds \rightarrow 0 \text{ as } R \rightarrow \infty \text{ uniformly in } t \in N.
    \end{align*}
\end{lemma}
\begin{proof}
    We rewrite 
\begin{align}
\label{decomp_sin}
    \int_\R F(t,s) \sin( Rs) \,ds  = \frac{1}{2i} \left( \int_\R F(t,s) e^{isR} \,ds -  \int_\R F(t,s) e^{-isR} \,ds \right).
\end{align}
Applying the substitution $s \mapsto s+ \frac{\pi}{R}$, we arrive at
\begin{align*}
    \int_\R F(t,s) e^{isR} \,ds = \int_\R F\left(t,s+ \frac{\pi}{R}\right) e^{isR} e^{i \pi} \,ds =  -\int_\R F\left(t,s+ \frac{\pi}{R}\right) e^{isR} \,ds.
\end{align*}
This shows
\begin{align*}
     \sup_{t \in N}\left|\left|\int_\R F(t,s) e^{isR} \,ds\right|\right| \leq \frac{1}{2} \sup_{t \in N} \int_\R \left|\left|F(t,s)-F\left(t,s+ \frac{\pi}{R}\right)\right|\right| \,ds.
\end{align*}
Similarly, one estimates the second integral in  (\ref{decomp_sin}) after taking its complex conjugate. By a Banach space-value version of the Kolmogorov compactness theorem (mimic the proof of \cite[2.32 Theorem]{adams}), we arrive at \begin{align}
\label{main_riemann_equation}
    \sup_{t \in N} \int_\R \left|\left|F(t,s)-F\left(t,s+ \frac{\pi}{R}\right)\right|\right| \,ds \rightarrow 0 \text{ as } R \rightarrow \infty,
\end{align}
which finishes the proof.
\end{proof}
\begin{lemma}
\label{compactness_criterion}
Let $F: \R \rightarrow X$ be $L^1$-integrable and $a\leq b$. Fix $\delta>0$. Suppose that\footnote{For $I \subset \R$, we set $\mathbf{1}_{I}(s) = 1$ for $s \in I$ and $\mathbf{1}_{I}(s) = 0$ for $s \notin I$.}
\begin{align*}
    \left\{s \mapsto \frac{F(t+s) - \mathbf{1}_{(-\delta,\delta)}(s) F(t)}{s}: t\in [a,b]\right\} 
\end{align*}
is a compact subset of $L^1(\R)$. Then, 
\begin{align*}
    F(t) = \lim_{R \rightarrow \infty} \frac{1}{2\pi}\int^{R}_{-R} e^{i t k} \hat{F}(k) \,dk \text{ uniformly in } t \in [a,b].
\end{align*}
\end{lemma}
\begin{proof}
Let $t\in [a,b]$.    Using (\ref{formula_invert}), (\ref{formula_one_dim}) and Proposition \ref{prop_properties} (i), we find
\begin{align*}
    S_R(F)(t) = \frac{1}{\pi}\int_0^\infty \frac{\sin(Rs)}{s}(F(t+s) +F(t-s)) \,ds = \frac{1}{\pi}\int_{-\infty}^\infty \frac{\sin(Rs)}{s}F(t+s) \,ds.
\end{align*}
Invoking Proposition \ref{prop_properties} (ii), we arrive at
    \begin{align*}
        \lim_{R \rightarrow \infty} \frac{1}{2\pi}\int^{R}_{-R} e^{is t} \hat{F}(s) \,ds  - F(t) =  &\lim_{R \rightarrow \infty}\frac{1}{\pi} \int_{-\infty}^\infty \frac{\sin(Rs)}{s} (F(t+s) - \mathbf{1}_{(-\delta,\delta)}(s) F(t))\,ds \\ 
    \end{align*}
    and the result follows from Lemma \ref{uniform_riemann}.
\end{proof}
We are now in the position to consider the special case of locally H\"older-continuous functions. Additionally, a mild localization condition is required which serves to apply dominated convergence. In the subsequent remark, we discuss further conditions and examples.
\begin{proposition}
\label{proposition_uniform_with_holder}
Let $F \in \Hloc(\R;X)$ be $L^1$-integrable. Furthermore, suppose that for every $a \leq b$ there exist  $\delta>0$ and  $g \in L^1(\R)$ such that
\begin{align}
\label{assumption_integrability}
   |s|^{-1}\sup_{t \in [a,b]} ||F(t+s)|| \leq g(s)
\end{align}
for all $|s|\geq \delta$.
Then:
\begin{align*}
    F(t) = \lim_{R \rightarrow \infty} \frac{1}{2\pi} \int^{R}_{-R} e^{is t} \hat{F}(s) \,ds  \text{ locally uniformly in } t \in \R.
\end{align*}  
\end{proposition}
\begin{proof}
Let $a\leq b$. Fix $\delta>0$.
    In virtue of Proposition \ref{compactness_criterion}, it suffices to show that $$M = \left\{s \mapsto \frac{F(t+s) - \mathbf{1}_{(-\delta,\delta)}(s) F(t)}{s}: t\in [a,b]\right\} $$ is compact in $L^1(\R)$. So, let $(t_n) \subset [a,b]$. Clearly, by compactness of $[a,b]$, there exists $(t_{n_k})$ and $t_0 \in [a,b]$ such that $t_{n_k} \rightarrow t_0$ as $k \rightarrow \infty$. By the H\"older-continuity of $F$ on $[a-\delta,b+\delta]$, there exists some $\alpha \in (0,1]$ and $C>0$ such that 
    \begin{align}
    \label{n_1_integrable_majorant}
       \left|\left| \frac{F(t_{n_k}+s)-\mathbf{1}_{(-\delta,\delta)}(s)F(t_{n_k})}{s}\right|\right| \leq C \left( \mathbf{1}_{(-\delta,\delta)}(s) \frac{1}{|s|^{1-\alpha}} + \mathbf{1}_{\{|s|\geq \delta\}}(s)\frac{1}{|s|}\sup_{k \in \mathbb{N}}||F(t_{n_k}+ s)||\right).
    \end{align}
    The function on the right hand side admits an integrable majorant thanks to  (\ref{assumption_integrability}). On the other hand, for fixed $s\in \R \setminus \{0\}$,
    \begin{align*}
        \frac{F(t_{n_k}+s)-\mathbf{1}_{(-\delta,\delta)}(s)F(t_{n_k})}{s} \rightarrow  \frac{F(t_{0}+s)-\mathbf{1}_{(-\delta,\delta)}(s)F(t_{0})}{s} \text{ as } k \rightarrow \infty
    \end{align*}
    as consequence of continuity of $F$. By (\ref{n_1_integrable_majorant}), the limit function lies also in $L^1(\R;X)$.
Dominated convergence now implies 
\begin{align*}
  \left( s \mapsto  \frac{F(t_{n_k}+s)-\mathbf{1}_{(-\delta,\delta)}(s)F(t_{n_k})}{s} \right)\rightarrow  \left(s \mapsto\frac{F(t_{0}-s)-\mathbf{1}_{(-\delta,\delta)}(s)F(t_{0})}{s}\right) \textrm{ in } L^1(\R;X) \text{ as } k \rightarrow \infty.
\end{align*}
This shows the compactness of $M$ in $L^1(\R;X)$.
\end{proof}

\begin{remark}
We briefly discuss the condition (\ref{assumption_integrability}). 
\begin{enumerate}
    \item If there exist $C>0$ and $\alpha >0$ such that
\begin{align*}
    ||F(s)|| \leq C (|s|+1)^{-\alpha}, \quad s \in \R,
\end{align*}
then (\ref{assumption_integrability}) holds. Indeed, for $a\leq b$, we find 
\begin{align*}
     |s|^{-1}\sup_{t \in [a,b]}||F(t+s)|| \leq C(a,b) (|s|+1)^{-1} \sup_{t \in [a,b]} (|t-s|+1)^{-\alpha}\leq C(a,b) (|s|+1)^{-1-\alpha},
\end{align*}
for all $|s|\geq \delta$ with $\delta = \max\{|a|,|b|\}+1$.
Clearly, the right hand side is an integrable function. 
As a non-trivial example to this condition, one considers a locally Hölder-continuous and integrable function with values $F(t) = \frac{1}{|k|^\frac{1}{2}}$ for $t \in \left[k- \frac{1}{2^{|k|}}, k+ \frac{1}{2^{|k|}}\right]$ and $k\in \mathbb{Z}\setminus \{0\}$. Such a function fulfills the decay assumption only with $0<\alpha \leq \frac{1}{2}$.
\item An example violating the condition in (i) is the function $F(t) = \sum_{n = 1}^\infty \frac{1}{n^2} e^{-|t-e^{n}|^2}$, $t \in \R$, since $F(e^n) \geq \frac{1}{n^2}$ for every $n\in \mathbb{N}$. The function $F$ is clearly Lipschitz-continuous  
 and  integrable. Furthermore, since the sum converges locally uniformly, we observe for $a\leq b$ and $t \in [a,b]$ that
\begin{align*}
   \sup_{t \in [a,b]} |F(t+s)| \leq \sum_{n = 0}^\infty \frac{1}{n^2} \sup_{t \in [a,b]}e^{-|t+s-e^{n}|^2} \leq F(b+s) + F(a+s)
\end{align*}
providing the condition (\ref{assumption_integrability}). Therefore,  by Proposition \ref{proposition_uniform_with_holder}, locally uniform Fourier inversion follows. This shows that the decay assumption in (i) is not necessary.
\item The last example teaches that if $||F||$ is integrable and eventually monotone, i.e. there exists some $C>0$ such that $||F(s)||$ is monotone for $|s|\geq C$, then $F$ satisfies (\ref{assumption_integrability}).
\item It would be interesting to find a function that is Dini-continuous and not H\"older-continuous for which locally uniform Fourier inversion fails.
\item As one only needs (\ref{main_riemann_equation}) to guarantee uniform convergence, one may replace the compactness condition by the slightly weaker assumption
\begin{align*}
 \sup_{t \in [a,b]} \int_\R  \left|\left| \frac{F(t+s) - \mathbf{1}_{(-\delta,\delta)}(s) F(t)}{s} -  \frac{F(t+s - h) - \mathbf{1}_{(-\delta,\delta)}(s+h) F(t)}{s+h}\right|\right| \,ds \rightarrow 0 \text{ as } h\downarrow 0
\end{align*}
in Proposition \ref{compactness_criterion}. %This criterion seems not to allow to relax the assumption (\ref{assumption_integrability}) significantly. 
\end{enumerate}
\end{remark}
\subsection{Uniformity in \texorpdfstring{$\mathbf{\R}^n$}{Lg} with odd \texorpdfstring{$n\geq 3$}{Lg}}
In this section, let $n = 2k+1$ with $k \in \mathbb{N}$. 
Let $x \in \R^n$ and $F: \R^n \rightarrow X$ be  $L^1$-integrable and $k$-times continuously differentiable. With the aid of Lemma \ref{formula_lemma}, we decompose
    \begin{align*}
        S_R(F) &= I_R(F) + II_R(F) + III(R) \textrm{ with } \\
        I_R(F)(x) 
        &=  C_{0,k}(n)\int_0^\infty \frac{\sin(Rr)}{\pi r}\bar{F}_x(r) \,dr, \quad 
        II_R(F)(x) =  \sum_{j = 1}^kC_{j,k}(n) \int_0^\infty \frac{\sin(Rr)}{\pi}r^{j-1} \left(\frac{d}{dr}\right)^j\bar{F}_x(r) \,dr, \\
         III_R(F)(x) &=  \int_0^\infty D_n^R(r) \bar{F}_x(r) (1-\eta_R(r)) r^{n-1} \,dr \\
         &\quad -\sum_{j = 0}^kC_{j,k}(n) \int_0^\infty \frac{\sin(Rr)}{\pi}r^{j-1} \left(\frac{d}{dr}\right)^j(\bar{F}_x(r) (1-\eta_R(r))) \,dr.
    \end{align*}
    Since, the constants $C_{j,k}(n)$ are defined iteratively, they can be explicitly determined. 
    \begin{remark}[Determination of $C_{0,k}(n)$]
    \label{remark_constant}As described in \cite{pinsky_book}, we use our knowledge on              pointwise inversion to show $C_{0,k}(n) = 2$ for each $k\in \mathbb{N}$.
        Considering the Schwartz function $F(x) = e^{-|x|^2}$ (or any other non-trivial Schwartz function), it is known that pointwise Fourier inversion holds since $F$ is in particular Dini-continuous. That is, for every $x \in \R$, it holds $S_R(F)(x) \rightarrow F(x)$ as $R\rightarrow \infty$. On the other hand, $r\mapsto \frac{\bar{F}_x(r) - \mathbf{1}_{(-\delta,\delta)}(r)\bar{F}_x(0)}{r}$ is $L^1$-integrable and therefore by the Riemann-Lebesgue Lemma, we conclude $\int_0^\infty \frac{\sin(Rr)}{\pi r}\bar{F}_x(r) \,dr \rightarrow \frac{1}{2} \bar{F}_x(0) = \frac{1}{2}F(x)$. Furthermore, $II_R(F)(x) \rightarrow 0$ as $R\rightarrow \infty$ since $ r \mapsto r^j \left(\frac{d}{dr}\right)^j\bar{F}_x(r)$ is $L^1$-integrable for $j = 1,...,k$.  Together with Lemma \ref{lemma_universel_zero}, we also find $III_R(F)(x) \rightarrow 0$ as $R\rightarrow \infty$. This implies $F(x) = \frac{1}{2}C_{0,k}(n)F(x)$ and hence $C_{0,k}(n) = 2$ as $F(x)>0$.
    \end{remark}
    We formulate the following compactness criterion.
\begin{proposition}
\label{odd_theorem}
Let $K\subset \R^n$ be compact. Let $F: \R^n \rightarrow X$ be $L^1$-integrable and $k$-times continuously differentiable, and suppose that \begin{align*}
   \left\{r \mapsto \frac{\bar{F}_x(r)-\mathbf{1}_{(0,\delta)}(r)F(x)}{r}: x \in K \right\}, \quad \left\{r \mapsto r^{j-1}\left(\frac{d}{dr}\right)^l \bar{F}_x(r): x \in K,\quad  l = 0,...,j \right\}, \quad j = 1,...,k,
\end{align*}
are compact subsets of $L^1((0,\infty);X)$ for some $\delta>0$.
Then, Fourier inversion holds uniformly in $K$.
\end{proposition}
\begin{proof}

Using $\bar{F}_x(0) = F(x)$, we first write
 \begin{align*}
     I_R(F)(x) - F(x) = \frac{2}{\pi}\int_{0}^\infty \frac{\sin(R r)}{r}\left(\bar{F}_x(r) - \mathbf{1}_{(0,\delta)}(r)F(x)\right) \,dr.
 \end{align*}
 By Lemma \ref{uniform_riemann} (extend the functions trivially to $\R$), it holds $I_R(F)(x) - F(x) \rightarrow 0$ as $R\rightarrow \infty$ uniformly in $x \in K$. Using again Lemma \ref{uniform_riemann}, compactness implies $II_R(x) \rightarrow 0$ as $R \rightarrow$ uniformly in $x \in K$. Finally, by Lemma \ref{lemma_universel_zero}, we find \begin{align}
 \label{term_by_product_rule}
     \sup_{x\in K}|III_R(F)(x)|\leq C\left( \sum_{j = 0}^k \sup_{x \in K} \sum_{l = 0}^j \int_R^\infty r^{j-1}  \left|\left(\frac{d}{dr}\right)^l\bar{F}_x(r)\right| \,dr + \frac{1}{R} ||F||_{L^1}\right)
 \end{align}
for a suitable constant $C>0$. The right hand side now converges to $0$ as $R \rightarrow \infty$ due to Kolmogorov's compactness theorem.
\end{proof}

\begin{proposition}
\label{prop_n_odd}
Assume that $F: \R^n \rightarrow X$ is $L^1$-integrable and $k$-times continuously differentiable. If for every compact $K\subset \R^n$ there exist $g \in L^1(\R^n)$ and $\delta>0$ such that
\begin{align}
\label{loc_cond_odd_dim}
    \sup_{\alpha\in \mathbb{N}_0^n, |\alpha| \leq k }\sup_{x \in K}|y|^{-n+k}\left|\left|\nabla^\alpha F(x+y)\right|\right|\leq g(y),
\end{align}
for all $|y|\geq \delta$. Then, Fourier inversion holds locally uniformly.
\end{proposition}
\begin{proof}
 Note that, since $F$ is continuously differentiable, it is locally H\"older-continuous.    Therefore, there exists a constant $C>0$ such that
    \begin{align*}
        \mathbf{1}_{(0,\delta)}(r)\left|\left|\frac{\bar{F}_x(r)-F(x)}{r}\right|\right| \leq \mathbf{1}_{(0,\delta)}(r)\frac{1}{r}\frac{1}{\omega_{n-1}} \int_{\mathbb{S}^{n-1}} ||F(x+r\omega) - F(x)|| \,d\omega \leq \mathbf{1}_{(0,\delta)}(r)C
    \end{align*}
    for every $r\geq 0$ and $x \in K$. Furthermore, we estimate\footnote{We denote positive  constants only depending on $\delta$ by $C(\delta)$.}, for $r \geq \delta$,
    \begin{align}
    \label{majorant_for_F_itself}
        \left|\left|\frac{1}{r}\bar{F}_x(r)\right|\right|\leq\sup_{x \in K} \frac{1}{\omega_{n-1}}\int_{\mathbb{S}^{n-1}} \frac{1}{r}||F(x+r\omega)|| \,d\omega \leq \frac{C(\delta)}{\omega_{n-1}} r^{n-1}\int_{\mathbb{S}^{n-1}} g(r\omega)  \,d\omega
    \end{align}
    where the right hand side is integrable in $[\delta,\infty)$ by (\ref{loc_cond_odd_dim}).
    By dominated convergence, argued as in the proof of Proposition \ref{proposition_uniform_with_holder}, we conclude that
    \begin{align*}
        \left\{r\mapsto \frac{\bar{F}_x(r)-\mathbf{1}_{(-\delta,\delta)}(r)F(x)}{r}\right\}
    \end{align*}
    is compact in $L^1((0,\infty);X)$. Similarly, one obtains the $L^1((0,\infty);X)$-compactness of 
    \begin{align*}
        \left\{r \mapsto r^{j-1}\left(\frac{d}{dr}\right)^l \bar{F}_x(r): x \in K, \quad l = 0,...,j\right\}, \quad j = 1,...,k,
    \end{align*}
    with the help of the observation that, for $r\geq 0$,
    \begin{align*}
        \left|\left|r^{j-1} \left(\frac{d}{dr}\right)^l \bar{F}_{x}(r)\right|\right| &\leq  \frac{1}{\omega_{n-1}} \int_{\mathbb{S}^{n-1}} \sup_{\alpha \in \mathbb{N}_0^k, |\alpha| \leq l} r^{j-1} ||\nabla^\alpha F(x+ r\omega)|| \,d\omega \\
        &\leq C(\delta)\Bigl(\mathbf{1}_{[\delta,\infty]}(r) \frac{1}{\omega_{n-1}} r^{n-1} \int_{\mathbb{S}^{n-1}} g(r\omega) \,d\omega \\
       & \qquad \qquad + \mathbf{1}_{(0,\delta)}(r) \sup_{\alpha \in \mathbb{N}_0^n, |\alpha| \leq l}\, \sup_{x \in K, y \in \mathbb{S}^{n-1}} ||\nabla^\alpha F(x+ ry)|| \Bigr), 
    \end{align*}
    for $j = 1,...,k$ and uniformly in $l = 0,...,j$ and $x \in K$, where we use the lemma about parameter integrals to pull derivatives into the surface integral.
   Putting everything together, the previous proposition yields locally uniform Fourier inversion.
\end{proof}

\begin{remark}[Comparison of conditions]
    For $n= 3$, we compare the conditions of Proposition \ref{prop_n_odd} with the ones from \cite[Proposition 4]{pinsky_paper} yielding convergence for fixed $x \in \R^3$. The localization condition is that   $|\cdot |^{-2}\nabla F(x-\cdot) \in L^1(\R^3)$. Instead of this pointwise assumption, we need an integrable majorant on $\mathbf{1}_{[\delta,\infty)}(\cdot)|\cdot|^{-2}(||F(x+\cdot)|| +\nabla ||F(x+\cdot))||$ uniformly for $x \in K$ and for any compact $K \subset \R^3$. The only further regularity assumption in \cite[Proposition 4]{pinsky_paper} is that $x$ is a Lebesgue point of $F$  which is covered by the property $F \in C^1(\R^3)$.
\end{remark}
\begin{remark}
    We consider exemplarily $n= 3$. To guarantee (\ref{majorant_for_F_itself}), it suffices to assume an integrable majorant on $\mathbf{1}_{[\delta,\infty)}(\cdot)|\cdot|^{-3}||F(x+\cdot)||$ uniformly in $x \in K$. On the other hand, writing out (\ref{term_by_product_rule}) gives rise to the term    $\int_\R \bar{F}_x(r)\left(\frac{1-\eta_R(r)}{r} + \eta_R'(r)\right)\,dr$
    for which we used the stronger localization assumption (\ref{loc_cond_odd_dim}). We do not see how to relax this assumption to the weaker one required for (\ref{majorant_for_F_itself}).
\end{remark}
\subsection{Uniformity in \texorpdfstring{$\mathbf{\R}^n$}{Lg} with even \texorpdfstring{$n\geq 2$}{Lg}}
As motivation, we consider first the case $n = 2$. Fix $x \in \R^2$. Let $F: \R^2\rightarrow X$ be such that $ r \mapsto \bar{F}_x(r)$ is integrable and differentiable with integrable derivative. Using Proposition \ref{prop_properties} (i), (iii), the fact that $\omega_1 = 2\pi$ and integrating by parts, we obtain
\begin{align*}
    S_R(F)(x) &= \int_0^\infty \bar{F}_x(r) RJ_1(Rr) \,dr\\
    &= \bar{F}_{x}(0) + \int_0^\infty \frac{d}{dr} (\bar{F}_x(r)) J_0(Rr) \,dr.
\end{align*}
Let $K\subset \R^2$ be compact. Since $J_0(Rr) \rightarrow 0$ as $R\rightarrow \infty$ for any $r>0$ and $\bar{F}_x(0) = F(x)$, dominated convergence implies $S_R(F)(x) \rightarrow F(x)$ uniformly in $x \in K$ whenever there exists some $g \in L^1((0,\infty))$ such that $\sup_{x \in K}|\frac{d}{dr}\bar{F}_x(r)| \leq g(r)$. For general even dimensions, we make the following statement.

\begin{proposition}
Let $n = 2k$ with $k\in \mathbb{N}$ and $K\subset \R^n$ be compact. Assume $F: \R^n \rightarrow X$ is $L^1$-integrable and $k$-times continuously differentiable. If there exist  $g \in L^1(\R^n)$ and $\delta >0$ such that 
\begin{align*}
  \sup_{\alpha\in \mathbb{N}_0^n, |\alpha| \leq k} \sup_{x \in K}|y|^{-n+k}||\nabla^\alpha F(x+y)||\leq g(y),
\end{align*}
for all $|y|\geq \delta$, then Fourier inversion holds uniformly in $K$.
\end{proposition}
\begin{proof}
   With the aid of Proposition \ref{prop_properties} (i) and Lemma \ref{formula_lemma}, we decompose
    \begin{align*}
        S_R(F)(x) = I_R(F)(x) + II_R(F)(x) + III_R(F)(x)
    \end{align*}
    with 
    \begin{align*}
         I_R(F(x)) &= C_{0,k-1}(n)\int_0^\infty RJ_1(Rr) \bar{F}_{x}(r) \,dr,\\
        II_R(F)(x) &=   \sum_{j = 1}^{k-1} C_{j,k-1}(n) \int_0^\infty RJ_1(Rr) r^j \left(\frac{d}{dr}\right)^j (\bar{F}_{x}(r)\eta_R(r))  \,dr, \\
        III_R(F)(x) &=  \frac{1}{2\pi}\int_0^\infty D_n^R(r) \bar{F}_x(r) (1-\eta_R(r)) r^{n-1} \,dr - C_{0,k-1}(n)\int_0^\infty RJ_1(Rr) (\bar{F}_{x}(r)(1-\eta_R(r)) \,dr\\
        &= III_R^1(F)(x) + III_R^2(F)(x),
    \end{align*}
    for universal constants $C_{j,k-1}(n)$, $j = 0,...,k-1$. As in Remark \ref{remark_constant}, one argues that $C_{0,k-1}(n) = 1$. Furthermore, integrating by parts once more and Proposition \ref{prop_properties} (i) yields
    \begin{align*}
        II_R(F)(x) = \sum_{j = 1}^{k-1} C_{j,k-1}(n) \Bigl(&j\int_0^\infty J_0(Rr) r^{j-1} \left(\frac{d}{dr}\right)^j (\bar{F}_{x}(r)\eta_R(r))  \,dr \\
        &+ \int_0^\infty J_0(Rr) r^{j} \left(\frac{d}{dr}\right)^{j+1} (\bar{F}_{x}(r)\eta_R(r))  \,dr\Bigr).
    \end{align*}

    Arguing as in Theorem \ref{odd_theorem}, one shows, for $r\geq 0$,
    \begin{align*}
        \sup_{x \in K}\left|\left|r^{j} \left(\frac{d}{dr}\right)^{l+1} (\bar{F}_{x}(r)\eta_R(r))\right|\right|  &\leq C(\delta)\Bigl(\mathbf{1}_{[\delta,\infty]}(r) \frac{1}{\omega_{n-1}} r^{n-1} \int_{\mathbb{S}^{n-1}} g(r\omega) \,d\omega \\
       & \qquad \qquad + \mathbf{1}_{(0,\delta)}(r) \sup_{\alpha \in \mathbb{N}_0^n, |\alpha| \leq l}\, \sup_{x \in K, y \in \mathbb{S}^{n-1}} ||\nabla^\alpha F(x+ ry)|| \Bigr),
    \end{align*}
    for $l = 0,...,j$ and $j = 0,...,k-1$, where the right-hand side is an integrable function on $(0,\infty)$. The introductory considerations for $n = 2$ with $j = 0$ give $I_R(F)(x) \rightarrow F(x)$ uniformly in $x \in K$. On the other hand, recalling again that $J_0(Rr) \rightarrow 0$ as $R\rightarrow \infty$ for any $r>0$, dominated convergence also implies $II_R(F)(x) \rightarrow 0$ as $R \rightarrow \infty$ uniformly in $x \in K$. Similarly, it follows $III_R^2(F)(x) \rightarrow 0$ as $R \rightarrow \infty$ uniformly in $x \in K$. Finally, by Lemma \ref{lemma_universel_zero}, we find $III_R^1(F)(x) \rightarrow 0$ as $R \rightarrow \infty$ uniformly in $x \in K$.
\end{proof}
%\subsection{Consequence for the wave equation}
%As pointed out in \cite{pinsky_paper}
\section{Applications to Laplace transforms}
\subsection{From Fourier to Laplace inversion}
\label{identify}
\begin{lemma}
   Let $0 \leq a\leq b$ and $R\geq 1$. Let $F: [0,\infty) \rightarrow X$  have exponential growth bound $\omega_0\in \R$. Fix $\omega>\omega_0$. Define the function $\tilde{F}:\R \rightarrow X$ by $$\Tilde{F}(s) = e^{-\omega s}\begin{cases} F(s), & \text{if } s\geq 0\\
    0 & \text{else}.\end{cases}$$
    Then,
\begin{align*}
F(t) = \lim\limits_{R\rightarrow \infty} \frac{1}{2\pi i} \int\limits_{\omega - i R}^{\omega+iR} e^{\lambda t} \mathcal{L}(F)(\lambda) \,d\lambda \textrm{ uniformly in } t \in [a,b]
\end{align*}   
    is equivalent to 
    \begin{align*}
        \Tilde{F}(t) = \lim_{R \rightarrow \infty} \frac{1}{2\pi }\int^{R}_{-R} e^{itk} \mathcal{F}(\Tilde{F})(k) \,dk \text{ uniformly in } t \in [a,b].
    \end{align*}
\end{lemma}
\begin{proof}
    Let $t \in [a,b]$ and $R\geq 1$. Rewriting
    \begin{align*}
        \frac{1}{2\pi i} \int\limits_{\omega - i R}^{\omega+iR} e^{\lambda t} \mathcal{L}(F)(\lambda) \,d\lambda &= \frac{1}{2\pi i } \int_{-iR}^{iR} e^{(\lambda + \omega)t} \int_0^\infty e^{-(\lambda + \omega)s} F(s) \,ds \,d \lambda \\
        &=\frac{1}{2\pi} e^{\omega t} \int_{-R}^{R} e^{ikt} \int_0^\infty e^{-iks} e^{-\omega s} F(s) \,ds \,d k \\
        &= e^{\omega t}\frac{1}{2\pi }\int^{R}_{-R} e^{itk} \mathcal{F}(\Tilde{F})(k) \,dk,
        \end{align*}
        shows the claim.
\end{proof}
As immediate consequence, we can translate Proposition \ref{proposition_uniform_with_holder} in terms of the Laplace transform and prove our main result:
\begin{proof}[Proof of Proposition \ref{prop_h_loc}]
    Let $\omega > \omega_0$ and  $\Tilde{F}$ as in Lemma \ref{identify}. Then $||\Tilde{F}(s)|| \leq e^{(\omega_0 - \omega)s}$ for $s>0$ and $\Tilde{F}(s) = 0$ for $s\leq 0$ which show that $\Tilde{F}$ satisfies (\ref{assumption_integrability}). Furthermore, we have $\Tilde{F}\in \Hloc(\R_+;X)$. This implies the assertion by Proposition \ref{proposition_uniform_with_holder} and Lemma \ref{identify}.
\end{proof}
\subsection{Application to \texorpdfstring{$C_0$-semigroups}{Lg}}
\label{section_semigroup}
 Let $A: D(A) \subset X \rightarrow X$ be  the generator of a $C_0$-semigroup $S(t)$ with exponential growth bound $\omega_0 < 0$. We recall the following well-known result on the complex inversion formula which is a consequence of Proposition \ref{know_result}, see  also \cite[5.15 Corollary]{nagel}).
\begin{corollary}
    If $\omega > \omega_0$ and $x \in D(A)$, then 
\begin{align}
S(t)x = \frac{1}{2\pi i} \lim\limits_{R\rightarrow \infty} \int\limits_{\omega - i R}^{\omega+iR} e^{\lambda t} (\lambda -A)^{-1}x \,d\lambda \textrm{ locally uniformly in } t \in (0,\infty)
\end{align}
in $X$-norm.
\end{corollary}
Using now Proposition \ref{prop_h_loc}, we may extend this result to Favard spaces.
For this purpose, let $\alpha \in (0,1]$. We define, see \cite[Section 5 b]{nagel}, the Favard space
\begin{align}
F_\alpha := \{x \in X: ||x||_{F_\alpha} < \infty\}, \quad ||x||_{F_\alpha} := \sup_{t>0}\frac{1}{t^\alpha}||S(t)x-x||_X.
\end{align}
Note that
$D(A) \subset F_\alpha \subset F_{\alpha'} \hookrightarrow X$ for any $0 < \alpha' \leq \alpha \leq 1$.\footnote{If $X$ is reflexive, then $D(A) = F^1$, c.f \cite[5.21 Corollary]{nagel}.}
%Considering $e^{-\alpha t}S(t)$ for $\alpha > \omega_0$ instead of $S(t)$ does not change $F_\alpha$ and we therefore may assume without loss of generality that $\omega_0 <0$.
The abstract H\"older space is given by
\begin{align*}
X_\alpha := \{x \in F_\alpha: \lim\limits_{t\downarrow 0} ||S(t)x-x||_{F_\alpha} = 0\}.
\end{align*}
With \cite[5.15 Theorem]{nagel}, we have following properties: both spaces $X_\alpha$ and $F_\alpha$ are Banach spaces with respect to $||\cdot||_{F_\alpha}$.
Furthermore, the restriction $S_{|F_\alpha}: [0,\infty) \rightarrow B(F_\alpha)$ is a semigroup of bounded operators and the restriction $S_{|X\alpha}: [0,\infty) \rightarrow B(X_\alpha)$ is a $C_0$-semigroup with generator $A_{|X_\alpha}$ and domain  $D(A_{|X_\alpha}) = X_{\alpha+1}$. 

\begin{corollary}
\label{cor_favard}
If $\omega > \omega_0$ and $x \in F_\alpha$, then
\begin{align}
\label{semigroup_inverse_with_zero}
S(t)x-e^{\omega_0 t}x = \frac{1}{2\pi i} \lim\limits_{R\rightarrow \infty} \int\limits_{\omega - i R}^{\omega+iR} e^{\lambda t} \left((\lambda -A)^{-1}-\frac{1}{\lambda- \omega_0}\right)x \,d\lambda \textrm{ locally uniformly in } t \in [0,\infty)
\end{align}
and
\begin{align}
\label{semigroup_inverse}
S(t)x = \frac{1}{2\pi i} \lim\limits_{R\rightarrow \infty} \int\limits_{\omega - i R}^{\omega+iR} e^{\lambda t} (\lambda -A)^{-1}x \,d\lambda \textrm{ locally uniformly in } t \in (0,\infty)
\end{align}
in $X$- and any $F_{\alpha'}$-norm with $\alpha' \in (0,\alpha)$.
\end{corollary}
\begin{proof}
Set $F(t) = S(t)x-e^{\omega_0 t}x$ and note that $F$ has growth bound $\omega_0$,  $F(0) = 0$ and $$\mathcal{L}(F)(\lambda) = (\lambda-A)^{-1}x - \frac{x}{\lambda- \omega_0}.$$ Obviously, $t \mapsto e^{\omega_0 t}x$ is locally H\"older-continuous and we consider $G(t) = S(t)x$. The function $G$ is H\"older-continuous from $[0,\infty)$ to $X$. In fact, for $t_2>t_1\geq 0$, we find
\begin{align*}
    ||G(t_1)-G(t_2)||_X \leq ||S(t_1)||_{X \rightarrow X} ||S(t_2- t_1)x - x ||_X \leq C_1 |t_2-t_1|^{\alpha} ||x||_{F_{\alpha}}
\end{align*}
for some constant $C_1>0$ independent of $t_2$ and $t_1$ and using that $\omega_0 <0$. 
Now, let $\alpha' \in (0,\alpha)$ and $t_2> t_1 \geq 0$.  With $\omega_0 <0$, we estimate
\begin{align*}
   ||G(t_2)-G(t_1)||_{F_{\alpha'}} &= \sup_{t>0}\frac{1}{t^{\alpha'}}||S(t_1)\left(S(t)(S(t_2-t_1)x- x)  - (S(t_2-t_1)x- x)\right)||_X \\
   &\leq C_1 |t_2-t_1|^{\alpha-\alpha'} \sup_{t_2-t_1 \geq t>0} \frac{1}{t^{\alpha}}||(S(t_2-t_1) (S(t)x- x)  -(S(t) x-x)||_X \\
   &\quad+C_1 |t_2-t_1|^{\alpha-\alpha'} \sup_{t \geq t_2-t_1} \frac{1}{(t_2-t_1)^{\alpha}}||S(t)(S(t_2-t_1)x- x)  - (S(t_2-t_1)x- x)||_X \\
   &\leq C_2|t_2-t_1|^{\alpha-\alpha'}||x||_{F_\alpha},
\end{align*}
%& \leq|t_2-t_1|^{\alpha-\alpha'} \sup_{t_2-t_1 \geq t>0} \frac{1}{t^{\alpha'}} \frac{1}{(t_2-t_1)^{\alpha-\alpha'}}||S(t)(S(t_2-t_1)x- x)  - (S(t_2-t_1)x- x)||\\
 %  &\quad+|t_2-t_1|^{\alpha-\alpha'} \sup_{t \geq t_2-t_1} \frac{1}{t^{\alpha'}} \frac{1}{(t_2-t_1)^{\alpha-\alpha'}}||S(t)(S(t_2-t_1)x- x)  - (S(t_2-t_1)x- x)|| \\
for some constants $C_{1,2}>0$ independent of $t_2$ and $t_1$. Similarly, one estimates $$||(e^{\omega_0 t_2}- e^{\omega_0 t_1})x||_{F^{\alpha'}} \leq C(a,b) (t_2-t_1)^{\alpha- \alpha'}||x||_{F^{\alpha}},$$ for every $0\leq a\leq t_1\leq t_2 \leq b$.
Therefore, (\ref{semigroup_inverse_with_zero}) follows by Proposition \ref{prop_h_loc} for $\omega>\omega_0$. Since $$\frac{1}{2\pi i} \lim_{R \rightarrow \infty} \int_{\omega -iR}^{\omega+iR} e^{\lambda t}\frac{x}{\lambda- \omega_0} \,d\lambda = e^{\omega_0 t}x,\quad \omega>\omega_0,$$ locally uniformly for $t \in (0,\infty)$, see e.g. \cite{arendt}, (\ref{semigroup_inverse}) follows from (\ref{semigroup_inverse_with_zero}).
\end{proof}
%
%The proof is very simple, since we only have to verify fixed $a>0$ that $\int\limits_{\omega - i R}^{w+iR} e^{\lambda t} (\lambda -A)^{-1}x \,d\lambda$ is a Cauchy sequence in $C([0,a], X)$. For this purpose, we recall that
%\begin{align*}
%||x||_{F_\alpha} \approx \sup\limits_{\Re (\lambda)>0}||\lambda^\alpha A (\lambda-A)^{-1}x||.
%\end{align*}
%Let $R,S \rightarrow \infty$ and $a>0$. We find, see again \cite{nagel},
%
%...
Assume that $F_\alpha \neq X_\alpha$. Then, the restriction $S_{|F_\alpha}$ is not strongly continuous on $F_\alpha$ and, on the other hand,
\begin{align}
\label{integral_laplace_cont}
t \mapsto  \frac{1}{2\pi i} \int\limits_{\omega - i R}^{\omega+iR} e^{\lambda t} \left((\lambda -A)^{-1}-\frac{1}{\lambda-\omega_0}\right)x \,d\lambda    
\end{align}
 is a continuous mapping from $[0,\infty)$ to $F_\alpha$ for every $R\geq 1$ and $x\in F_\alpha$. Therefore, the convergence in (\ref{semigroup_inverse_with_zero}) cannot hold for every $x \in F_\alpha$ with respect to the $F_\alpha$-norm. 
 
We justify the continuity of (\ref{integral_laplace_cont}). For this purpose, we fix $x\in X$ and set $y(\lambda) = (\lambda-A)^{-1}x$ which is in particular a continuous function in $\lambda$ from $\rho(A)$ to $X$.  Let $\Re(\lambda) = \omega>\omega_0$. We first estimate
\begin{align*}
    ||S(t)y(\lambda) - y(\lambda)||_X = ||\int_0^t S(s)Ay(\lambda)\,ds||_X \leq C_1 t ||Ay(\lambda)||_X 
\end{align*}
for some constant $C_1>0$. Furthermore, we use $$Ay(\lambda) = \lambda y(\lambda) -x = \lambda \int_0^\infty e^{-(\lambda -\omega_0)s }e^{-\omega_0 s}F(s)\,ds + \left(\frac{\lambda}{\lambda-\omega_0}-1\right)x$$ with $F$ as in the proof of Corollary \ref{cor_favard} to estimate
\begin{align*}
  ||Ay(\lambda)||_X \leq  C_2\left(|\lambda| + \frac{|\omega_0|}{|\lambda-\omega_0|}\right) ||x||_X
\end{align*}
for some $\omega$-dependent constant $C_2 >0$.
We conclude that
there exists some $\omega$-dependent constant $C_3>0$ such that
\begin{align*}
 ||y(\lambda)||_{F_\alpha} &\leq \sup_{t>0} \left(\frac{1}{t} ||S(t)y(\lambda)-y(\lambda)||_X\right)^\alpha||S(t)y(\lambda)- y(\lambda)||_X^{1-\alpha}
 \leq C_3 ||y(\lambda)||_X^{1-\alpha}\left(1+|\lambda|\right)^\alpha ||x||_X^\alpha.
\end{align*}
 This yields the integrability of $\lambda \mapsto \left((\lambda-A)^{-1} - \frac{1}{\lambda-\omega_0}\right)x$ in $F_\alpha$ on every $\{\omega + ik: k\in [-R,R]\}$ with  $R\geq1$. Thereby, the continuity of (\ref{integral_laplace_cont}) follows for every $R\geq 1$ by the lemma about parameter integrals.
If $S$ is a $C_0$-group, then there exists $x\in F^\alpha$ such that (\ref{semigroup_inverse}) also fails immediately by the same argument.

\paragraph*{Oulook.} The considerations of this section beg the question of whether we get convergence in $F_\alpha$-norm for every $x \in X_\alpha$. More generally, one may ask whether for every Banach space $X$ and $C_0$-semigroup (or $C_0$-group) $S$, there exists a (largest) non-empty Banach space $Y \hookrightarrow X$ such that for every $x \in Y$ the convergence in (\ref{semigroup_inverse_with_zero}) (or (\ref{semigroup_inverse})) holds in $Y$-norm. How does $X$ or $S$ have to be specified and how does $Y$ look like?  We leave these questions open for future research. 
As an example of such a specification, we mention:  for every $C_0$-semigroup and if $X$ is a UMD-space, then locally uniform convergence of the complex inversion formula is positively answered choosing $Y = X$ as shown in \cite{UMD}.

For another straightforward example, consider a $C_0$-semigroup $S$ with generator $A$ such that $\textrm{ker}(A) \neq \emptyset$. By the Hille-Yosida theorem, $S$ has  necessarily growth bound $\omega_0\geq 0$ and one chooses $Y = \textrm{ker}(A)$ equipped with the norm $||\cdot||_X$ as non-trivial choice for the general question.

\appendix
\section{Function spaces}
\label{appendix}
Let $I \in \{ \R^n, \R_+\}$ with $\R_+ = [0,\infty)$ and $n\in \mathbb{N}$. We introduce the following spaces
\begin{align*}
&\Lip(I;X) := \{f: I \rightarrow X \textrm{ is Lipschitz-continuous}\}, \\
&\Hloc(I;X) := \{f: I \rightarrow X \textrm{ is locally H\"older-continuous}\}, \\
&\textrm{Din}(I;X) := \{f: I \rightarrow X \textrm{  is Dini-continuous}\}, \\
&C(I;X) := \{f: I \rightarrow X \textrm{ is continuous}\}, 
\end{align*}
where we call $F: I \rightarrow X$ Dini-continuous at point $t_0 \in I$ if there exists some $\delta_0>0$ such that
\begin{align*}
\int\limits_{(-\delta_0,\delta_0)^n} \frac{||F(t_0 + h)- F(t_0)||}{|h|} \,dh < \infty.
\end{align*}
The function $F$ is called Dini-continuous if it is Dini-continuous in every point.
We call a function $F:I \rightarrow X$ locally H\"older-continuous, if for every compact subset $K$ of $I$ there exist an $\alpha_0 = \alpha_0(K) \in (0,1]$ and $C = C(K)>0$ such that
\begin{align*}
 ||F(x)-F(y)||\leq C|x-y|^{\alpha_0} \textrm{ for all } x,y\in K.
\end{align*} Obviously, we have the inclusions
\begin{align*}
\Lip(I;X) \subset \Hloc(I;X) \subset \Dini(I;X) \subset C(I;X).
\end{align*}
\bibliographystyle{abbrv}
\bibliography{refs}
\end{document}